\pdfoutput=1
\documentclass[a4paper,
   abstracton]{scrartcl}
\usepackage{amsmath, amssymb, amsthm, amsfonts}
\usepackage{xcolor}

\usepackage{listings}
\usepackage{rotating}

\usepackage[ngerman,british]{babel}

\usepackage{mathdots}

\usepackage[ruled, vlined]{algorithm2e}

\usepackage{xspace}

\usepackage{enumerate}

\usepackage{adjustbox}

\usepackage[style=alphabetic,
   sorting=nty,
   backend=bibtex,
   maxbibnames=99,
   doi=false,
   isbn=false,
   backref=false]{biblatex}
\usepackage[tracking, kerning, spacing]{microtype}
\microtypecontext{spacing=nonfrench}

\usepackage{tikz}
\usepackage{tikz-3dplot}
\usetikzlibrary{matrix,fit,positioning,shapes.geometric,patterns,backgrounds}
\usetikzlibrary{decorations.pathreplacing}
\usetikzlibrary{arrows,calc}
\tikzstyle{bigbox} = [draw=blue!50, thick, rounded corners, rectangle]
\tikzset{
>=stealth'
}
\tikzset{
    old inner xsep/.estore in=\oldinnerxsep,
    old inner ysep/.estore in=\oldinnerysep,
    double circle/.style 2 args={
        circle,
        old inner xsep=\pgfkeysvalueof{/pgf/inner xsep},
        old inner ysep=\pgfkeysvalueof{/pgf/inner ysep},
        /pgf/inner xsep=\oldinnerxsep+#1,
        /pgf/inner ysep=\oldinnerysep+#1,
        alias=sourcenode,
        append after command={
        let     \p1 = (sourcenode.center),
                \p2 = (sourcenode.east),
                \n1 = {\x2-\x1-#1-0.5*\pgflinewidth}
        in
            node [inner sep=0pt, draw, circle, minimum width=2*\n1,at=(\p1),#2] {}
        }
    },
    double circle/.default={2pt}{blue}
}

\defbibheading{custombibheading}[\bibname]{%
\chapter{#1}
}

\newcommand{\inj}{\hookrightarrow}
\newcommand{\lat}{N}
\newcommand{\latv}{M}

\newcommand{\scalp}[1]{\langle #1 \rangle}
\newcommand{\surj}{\twoheadrightarrow}

\newcommand{\CC}{\mathbb C}

\newcommand{\OO}{\mathcal O} 
\newcommand{\QQ}{\mathbb Q}

\newcommand{\ZZ}{\mathbb Z}

\definecolor{gold}{rgb}{0.9,0.77,0}      

\definecolor{tomato}{rgb}{1.0,.388,.278}      
\definecolor{DarkSalmon}{rgb}{.914,.588,.478} 
\definecolor{SaddleBrown}{rgb}{0.545,.271,.075}      

\definecolor{dirtyyellow}{rgb}{0.8,0.67,0}
\definecolor{yellow}{rgb}{0.8,0.67,0}
\definecolor{dirtyred}{rgb}{0.89,0.1,0.1}
\definecolor{red}{rgb}{0.89,0.1,0.1}
\definecolor{darkred}{rgb}{0.79,0.05,0.05}
\definecolor{dirtygreen}{rgb}{0.1,0.89,0.1}
\definecolor{green}{rgb}{0.1,0.89,0.1}
\definecolor{darkgreen}{rgb}{0.05,0.79,0.05}
\definecolor{dirtyblue}{rgb}{0.1,0.1,0.89}

\definecolor{klaus}{rgb}{.914,.588,.478} 
\newcommand{\define}[1]{{\it #1}}

\definecolor{quec}{rgb}{0.8,.000,.278}      
\definecolor{ansc}{rgb}{.200,0.7,.400}     
\definecolor{idec}{rgb}{.400,.000,.700}     
\definecolor{todo}{rgb}{.600,.600,.000}     

\newtheoremstyle{question}
  {2pt}   
  {2pt}   
  {\color{quec}\normalfont}  
  {0pt}       
  {\bfseries} 
  {.}         
  {5pt plus 1pt minus 1pt} 
  {}          
\newtheoremstyle{idea}
  {2pt}   
  {2pt}   
  {\color{idec}\normalfont}  
  {0pt}       
  {\bfseries} 
  {.}         
  {5pt plus 1pt minus 1pt} 
  {}          
\newtheoremstyle{answer}
  {2pt}   
  {2pt}   
  {\color{ansc}\normalfont}  
  {0pt}       
  {\bfseries} 
  {.}         
  {5pt plus 1pt minus 1pt} 
  {}          
\newtheoremstyle{todo}
  {2pt}   
  {2pt}   
  {\color{todo}\normalfont}  
  {0pt}       
  {\bfseries} 
  {.}         
  {5pt plus 1pt minus 1pt} 
  {}          
\theoremstyle{idea}

\theoremstyle{question}

\theoremstyle{answer}

\theoremstyle{todo}

\usepackage{aliascnt}

\usepackage{titletoc}
\usepackage[toctitles]{titlesec}
\usepackage{hyperref}

\definecolor{hrcitecolor}{rgb}{0,0,0.6}
\definecolor{hrlinkcolor}{rgb}{0,0.3,0}
\hypersetup{
   colorlinks   = true,
   citecolor    = hrcitecolor,
   linkcolor    = hrlinkcolor
}

\newcounter{internal}[section]

\newaliascnt{intcor}{internal} 
\newaliascnt{intconj}{internal} 
\newaliascnt{intlemma}{internal} 
\newaliascnt{intdef}{internal} 
\newaliascnt{intex}{internal} 
\newaliascnt{intprop}{internal} 
\newaliascnt{intrem}{internal} 
\newaliascnt{intthm}{internal}

\newcommand*{\dupcntr}[2]{%
   \expandafter\let\csname c@#1\expandafter\endcsname\csname c@#2\endcsname
}
\dupcntr{figure}{internal}
\dupcntr{algocfline}{internal}

\theoremstyle{definition}
\newtheorem*{notation}{Notation}

\newtheorem{lemma}[intlemma]{Lemma}
\newtheorem{definition}[intdef]{Definition}
\newtheorem{example}[intex]{Example}
\newtheorem{prop}[intprop]{Proposition}

\theoremstyle{remark}
\newtheorem{remark}[intrem]{Remark}

\theoremstyle{plain}
\newtheorem{theorem}[intthm]{Theorem}

\DeclareMathOperator{\cone}{cone} 
\DeclareMathOperator{\conv}{conv} 
\DeclareMathOperator{\ext}{ext} 
\DeclareMathOperator{\Ext}{Ext} 
\DeclareMathOperator{\GL}{GL} 
\DeclareMathOperator{\Hom}{Hom} 
\DeclareMathOperator{\img}{im} 
\DeclareMathOperator{\incoming}{in_{\resg}} 
\DeclareMathOperator{\interior}{int} 
\DeclareMathOperator{\Spec}{Spec} 
\DeclareMathOperator{\Tor}{Tor} 
\DeclareMathOperator{\vertex}{\nu} 

\newcommand{\canonical}{K}
\newcommand{\cqs}{X}
\newcommand{\divi}{D}
\newcommand{\divz}{D'}
\newcommand{\field}{\CC}
\newcommand{\matlis}[1]{(#1)^{\vee}}
\newcommand{\mingens}[1]{G(#1)} 
\newcommand{\module}[1]{H_{#1}} 
\newcommand{\polyhedron}[1]{P_{#1}} 
\newcommand{\resg}{\mathcal Q}
\newcommand{\resgvert}{{\resg}_0}
\newcommand{\resgar}{{\resg}_1}
\newcommand{\rundiv}{W}
\newcommand{\short}[1]{\field\{#1\}}
\newcommand{\sigv}{\sigma^{\vee}}

\DeclareMathOperator{\abelow}{T} 
\DeclareMathOperator{\below}{E} 
\DeclareMathOperator{\id}{id}
\DeclareMathOperator{\tor}{tor}
\DeclareMathOperator{\universal}{link} 

\title{$\Ext$ and $\Tor$ on two-dimensional cyclic quotient singularities}
\author{Lars Kastner
\thanks{
The author is supported by the DFG (German research foundation) priority program SPP 1489.
}
\\Freie Universit\"at Berlin}
\date{}

\begin{document}

\maketitle
\setcounter{tocdepth}{1}

Given two torus invariant Weil divisors $\divi$ and $\divz$ on a
two-dimensional cyclic quotient singularity $\cqs$, the groups
$\Ext^i_{\cqs}(\OO(\divi),\OO(\divz))$, $i>0$, are naturally $\ZZ^2$-graded. We
interpret these groups via certain combinatorial objects using methods from
toric geometry. In particular, it is enough to give a combinatorial description
of the $\Ext^1$-groups in the polyhedra of global sections of the Weil divisors
involved.  Higher $\Ext^i$-groups are then reduced to the case of $\Ext^1$ via
a quiver. We use this description to show that $\Ext^1_{\cqs}(\OO(D),\OO(K-D'))
= \Ext^1_{\cqs}(\OO(D'),\OO(K-D))$, where $K$ denotes the canonical divisor on
$\cqs$. Furthermore, we show that $\Ext^{i+2}_{\cqs}(\OO(D),\OO(D'))$ is the Matlis dual of
$\Tor_{i}^{\cqs}(\OO(D),\OO(D'))$.

\section{Preliminaries}\label{prelim}
Let us start by recalling the basic definitions and notation. For
toric geometry we will follow \cite{cls} closely.

Given two coprime positive integers $0<q<n$, we can define a 2-dimensional
cyclic quotient singularity $\cqs$ by taking the quotient of $\CC^2$ by the
action of the cyclic subgroup of $\GL(2,\CC)$ generated by
\[
\left(
\begin{array}{cc}
\xi_n & 0\\
0 & \xi_n^q
\end{array}
\right),
\]
where $\xi_n$ denotes an $n$-th root of unity. Any 2-dimensional cyclic
quotient singularity arises in this manner.

The equivalent definition in terms of toric geometry is as follows: Let $\lat$
be a two-dimensional lattice and let $\latv=\Hom_{\ZZ}(\lat,\ZZ)$ be the dual lattice.
Now let $\lat_{\QQ}:=\lat\otimes_{\ZZ}\QQ$ and $\latv_{\QQ}$ be the associated
$\QQ$-vector spaces. We identify $\lat$ and $\latv$ with $\ZZ^2$ by choosing the usual
scalar product as pairing.

\begin{definition}
Given two integers $0<q<n$ such that $\gcd(q,n)=1$, we define the
\define{cyclic quotient singularity} $\cqs$ to be
\[
\cqs\ :=\ \Spec\field[\sigv\cap\latv],
\]
where $\sigma^\vee$ denotes the dual cone of
\[
\sigma\ :=\ \cone\left(\rho^0:=(1,0),\rho^1:=(-q,n)\right)\ \subseteq\ \lat_{\QQ}.
\]
By $R$ we denote the coordinate ring $\field[\sigv\cap\latv]$ of $\cqs$.
\end{definition}

Let us introduce a running example.
\begin{example}\label{n7q3:hilbert_basis}
Let $n=7$ and $q=3$.  Then the Hilbert basis of the dual cone $\sigma^\vee$ has four
elements, indicated by the dots in the picture.
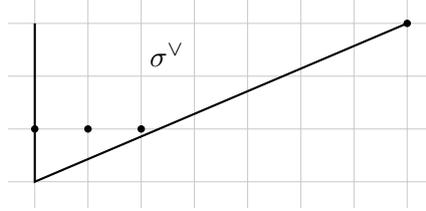
\[
\begin{tikzpicture}[scale=.7]
\draw[step=1.0,black!20,thin] (-0.5,-0.5) grid (7.5,3.5);
\draw[thick] (0,3) -- (0,0) -- (7,3);
\fill (0,1) circle (2pt);
\fill (1,1) circle (2pt);
\fill (2,1) circle (2pt);
\fill (7,3) circle (2pt);
\node[anchor=south east] at (3,2) {$\sigma^\vee$};
\end{tikzpicture}
\]
Thus $R=\field[x^{[0,1]},x^{[1,1]},x^{[2,1]},x^{[7,3]}]$ or, if we label the axes
$x$ and $y$, $R=\field[y,xy,x^2y,x^7y^3]$.
\end{example}

\begin{remark}
There is a close relationship between the Hilbert basis of $\sigma^\vee$ and
the continued fraction expansion of $\frac{n}{n-q}$ discovered by
Riemenschneider (\cite{cqsgens}).  This has already lead to a very fruitful
discussion of their deformation theory in terms of chains representing zero
(\cite{stevens:91, christophersen:91}) and so-called p-resolutions
(\cite{altmann:98}). The connection of the $\Ext$ functor with the continued
fraction expansion of $\frac{n}{n-q}$ is part of \cite{thesis}.
\end{remark}

Any torus invariant Weil divisor is an integer linear combination of the orbits
corresponding to the rays of $\sigma$, denoted as $[\rho^0]$ and $[\rho^1]$. We
write $\divi=a_0[\rho^0] + a_1[\rho^1]$, $a_i\in\ZZ$.

\begin{notation}
Throughout this article, we will omit the $\OO$, for example we will write
$\Ext^i_{\cqs}(\divi,\divz)$ for $\Ext^i_{\cqs}(\OO(\divi),\OO(\divz))$.
\end{notation}

In order to study the
modules $\Ext^i_{\cqs}(\divi,\divz)$ we may instead study the modules
$\Ext^i_{R}(\Gamma(\cqs,\divi),\Gamma(\cqs,\divz))$, since $X$ is
affine. In the toric language the global sections of a torus invariant Weil
divisor are given by the section polyhedron.

\begin{definition}
For a torus invariant Weil divisor $\divi=a_0[\rho^0] + a_1[\rho^1]$, the
\define{section polyhedron} is given by
\[
\polyhedron{\divi}\ :=\ \{u\in\latv_{\QQ}\ |\ \scalp{u,\rho^i}\ge -a_i\}. 
\]
\end{definition}

The lattice points of the section polyhedron correspond to the homogeneous
global sections of $\OO(\divi)$. We use this to describe $H^0(\cqs,\OO(\divi))$ as a
divisorial ideal over $R$:
\[
\module{\divi}\ :=\ H^0(\cqs,\OO(\divi))\ =\ \bigoplus_{u\in\polyhedron{\divi}\cap\latv}\field\cdot\chi^u\ \subseteq\ \field[\latv].
\]
Keep in mind that $\divi$ is not necessarily Cartier. Hence the summation
$\divi+\divz$ of divisors does not translate to the multiplication of the
ideals, and in general we have
\[
\module{\divi+\divz}\subsetneq\module{\divi}\cdot\module{\divz}.
\]
Because $\sigma^\vee$ is simplicial, every $\polyhedron{\divi}$ will just be
the Minkowski sum of $\sigma^\vee$ and a rational vector
$\vertex(\divi)\in\QQ^2$. We write
\[
\polyhedron{\divi}\ =\ \vertex(\divi)+\sigma^\vee.
\]

The (minimal) homogeneous generators of the $R$-module correspond exactly to
the lattice points on the compact edges of
$\conv(\polyhedron{\divi}\cap\latv)$. Denote these lattice points as 
\[
\mingens{\divi}\ :=\ \{u\in\latv\ |\ u\mbox{ lies on a compact edge of }\conv(\polyhedron{\divi}\cap\latv)\}\ \subseteq\ \latv.
\]
\begin{remark}\label{gens:ineq}
If $\mingens{\divi}=\{u^0,\ldots,u^r\}$ we can sort these lattice points from
left to right, such that
\[
\scalp{u^0,\rho^0} < \scalp{u^1,\rho^0} <\ldots <\scalp{u^r,\rho^0}
\]
and
\[
\scalp{u^r,\rho^1} < \scalp{u^{r-1},\rho^1} <\ldots <\scalp{u^0,\rho^1}.
\]
\end{remark}

\begin{notation}
Throughout this paper we will use the following shorthand notation: Let $P$ be
a subset of $\latv_{\QQ}$, then we denote by $\short{P}$ the $R$-module
\[
\short{P}\ :=\ \bigoplus_{u\in P\cap\latv}\field\cdot\chi^u,
\]
with the multiplication
\[
x^w\cdot\chi^u\ :=\
\begin{cases}
\chi^{u+w} & u+w\in P\\
0 & \mbox{else}
\end{cases}.
\]
Note that only with this multiplication $\short{P}$ becomes a $\latv$-graded
$R$-module. Furthermore, this is not well-defined for every subset $P$ of
$\QQ^2$, rather one needs ``convexity'' of $P$ with respect to $\sigma^\vee$.
Thus, we require for $p\in P$ with $u+p\in P$ for $u\in\sigma^\vee$ that
$\lambda u+p\in P$ for $0\le \lambda \le 1$, which is actually a little stronger than needed.

As an example note that $\module{\divi} = \short{\polyhedron{\divi}}$.
\end{notation}

The class group of $\cqs$ is $\ZZ/n\ZZ$. By $0=E^0,\ldots, E^{n-1}$ we denote
the torus invariant divisors $E^i:=-i\cdot[\rho^0]$.  These divisors form a
system of representatives for the class group of $\cqs$.  Furthermore, let
$\canonical=-[\rho^0]-[\rho^1]$ be the canonical divisor. One can now calculate
the vertex of the canonical divisor to be
\[
\vertex(\canonical)\ =\ [1,\frac{q+1}{n}].
\]

\begin{example}\label{cqs:e1e3:picture}
We will consider the divisors \textcolor{red}{$E^1$} and
\textcolor{green}{$E^3$}. First we draw the corresponding polyhedra of global
sections.
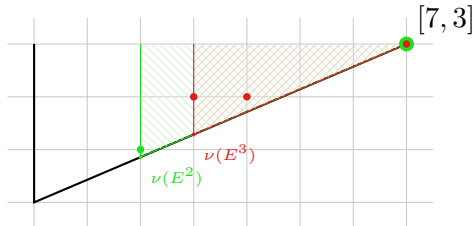
\[
\begin{tikzpicture}[scale=.7]
\draw[step=1.0,black!20,thin] (-0.5,-0.5) grid (7.5,3.5);
\draw[thick] (0,3) -- (0,0) -- (7,3);
\fill[pattern color=green!20, pattern=north west lines] (2,3) -- (2,6/7) -- (7,3) -- cycle;
\fill[pattern color=red!20, pattern=north east lines] (3,3) -- (3,9/7) -- (7,3) -- cycle;
\draw[color=green] (2,3) -- (2,6/7) -- (7,3);
\fill[color=green] (2,6/7) node[anchor=north west] {\tiny{$\vertex(E^2)$}} circle (1pt);
\draw[color=red] (3,3) -- (3,9/7) -- (7,3);
\fill[color=red] (3,9/7) node[anchor=north west] {\tiny{$\vertex(E^3)$}} circle (1pt);
\fill[color=green] (2,1) circle (2pt);
\fill[color=green] (7,3) node[color=black, anchor=south west] {$[7,3]$} circle (4pt);
\fill[color=red] (3,2) circle (2pt);
\fill[color=red] (4,2) circle (2pt);
\fill[color=red] (7,3) circle (2pt);
\end{tikzpicture}
\]
Then we compute the generators of $\module{\textcolor{green}{E^2}}$ and
$\module{\textcolor{red}{E^3}}$. They are depicted as dots of the respective
colors. 
\end{example}

Our approach is to resolve such a divisorial ideal $\module{\divi}$
projectively in a Taylor-like fashion (\cite{taylor:66}). Choosing a minimal
generating set of syzygies, we arrive at a short exact sequence, which may be
seen as a slight generalization of cellular resolutions
(\cite[Ch. 4]{sturmfelsmiller}) to the coordinate rings of cyclic quotient
singularities.

\section{Resolving torus invariant divisors on CQS}\label{resolution}
A major obstacle, stemming from the
non-regularity of $R$, is the infiniteness of the complexes.
It turns out that the syzygies of the generators of $\module{\divi}$ are
isomorphic to a direct sum of divisorial ideals, i.e. global sections of other
torus invariant divisors. Finiteness of the class group yields that we can
encode all the information needed to freely resolve any $\module{\divi}$ in a
finite quiver $\resg$ with edges labelled by elements of $\latv$, effectively
overcoming the previously mentioned obstacle. This results in a recursive
formula for both $\Ext^i$ and $\Tor_i$, $i>1$, in terms of $\Ext^1$ and
$\Tor_1$, respectively.

Denote by $\{u^0,\ldots, u^r\}$ the generators $\mingens{\divi}$ of $\divi$.
Now take $\pi$ to be the canonical surjection 
\[
\bigoplus_{i=0}^rR[-u^i]\surj \module{\divi},\ e^i\mapsto x^{u^i}.
\]
For every pair $u^i$ and $u^{i+1}$ of consecutive generators, we can build an injective map into the kernel of $\pi$ in the following way:
\[
(x^{u^{i-1}})\cap (x^{u^i})\to \oplus_{i=0}^rR[-u^i],\ x^u\mapsto x^{u-u^i} e^i-x^{u-u^{i-1}}e^{i-1}.
\]
Denote by $\iota$ the map from the direct sum $\oplus_{i=1}^r (x^{u^{i-1}})\cap
(x^{u^i})$ of these ideals into $\oplus_{i=0}^rR[-u^i]$.

The following proposition may be seen as a generalization of \cite[Prop
3.1]{sturmfelsmiller}, which states that every monomial ideal in the polynomial
ring with two variables can be resolved freely in a short exact sequence, to
the singular case.

\begin{prop}\label{ses}
The sequence
\[
0 \to 
\bigoplus_{i=1}^r ((x^{u^{i-1}})\cap (x^{u^i})) \stackrel{\iota}{\inj} 
\bigoplus_{i=0}^rR[-u^i] \stackrel{\pi}{\surj} 
\module{\divi} \to 
0
\]
is exact. Furthermore, the fractional ideals $(x^{u^{i-1}})\cap (x^{u^i})$ are divisorial.
\end{prop}

Let us illustrate this proposition in the example.
\begin{example}\label{cqs:free_res:ex:res}
In our example $D=E^3$ is
an ideal of $R$, generated by $x^3y^2$, $x^4y^2$ and $x^7y^3$. We obtain
the following sequence
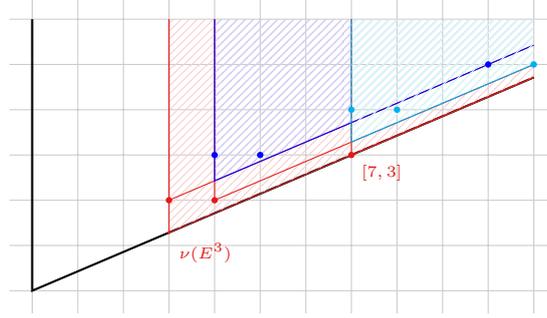
\[
\begin{tikzpicture}[scale=0.6]
\draw[step=1.0,black!20,thin] (-0.5,-0.5) grid (11.5,6.5);
\draw[thick] (0,6) -- (0,0) -- (11,3+4*3/7);
\fill[pattern color=red!20, pattern=north east lines] (3,6) -- (3,9/7) -- (11,3+4*3/7) -- (11,6) -- cycle;
\draw[color=red] (3,6) -- (3,9/7) -- (11,3+4*3/7);
\draw[color=red] (3,6) -- (3,2) -- (11,2+8*3/7);
\draw[color=red] (4,6) -- (4,2) -- (11,2+7*3/7);
\draw[color=red] (7,6) -- (7,3) -- (11,3+4*3/7);
\fill[pattern color=blue!20, pattern=north east lines] (4,6) -- (4,2+3/7) -- (11,2+8*3/7) -- (11,6) -- cycle;
\draw[color=blue] (4,6) -- (4,2+3/7) -- (11,2+8*3/7);
\fill[pattern color=cyan!20, pattern=north east lines] (7,6) -- (7,2+3*3/7) -- (11,2+7*3/7) -- (11,6) -- cycle;
\draw[color=cyan] (7,6) -- (7,2+3*3/7) -- (11,2+7*3/7);
\fill[color=red] (3,2) circle (2pt);
\fill[color=red] (4,2) circle (2pt);
\fill[color=red] (7,3) circle (2pt);
\fill[color=blue] (4,3) circle (2pt);
\fill[color=blue] (5,3) circle (2pt);
\fill[color=blue] (10,5) circle (2pt);
\fill[color=cyan] (7,4) circle (2pt);
\fill[color=cyan] (8,4) circle (2pt);
\fill[color=cyan] (11,5) circle (2pt);
\fill[color=red] (3,9/7) node[anchor=north west] {\tiny{$\vertex(E^3)$}} circle (1pt);
\fill[color=red] (7,3) node[anchor=north west] {\tiny{$[7,3]$}} circle (1pt);
\end{tikzpicture}
\]

\[
0\to
[\textcolor{blue}{(x^{[3,2]})\cap (x^{[4,2]})}]\oplus
[\textcolor{cyan}{(x^{[4,2]})\cap (x^{[7,3]})}]
\inj
R[-[3,2]]\oplus R[-[4,2]]\oplus R[-[7,3]]
\surj \module{E^3} \to 0
\]
One immediately recognizes the summands of the first term as divisorial ideals, i.e.
\[
[\textcolor{blue}{(x^{[3,2]})\cap (x^{[4,2]})}]\oplus
[\textcolor{cyan}{(x^{[4,2]})\cap (x^{[7,3]})}]
=
\textcolor{blue}{\module{E^1[-[3,2]]}}\oplus\textcolor{cyan}{\module{E^3[-[4,2]]}}.
\]
\end{example}

We begin the proof of \autoref{ses} by showing that all homogeneous
elements of $\ker{\pi}$ are in the image of $\iota$.
\begin{lemma}\label{homogeneous:onto}
Let $\underline{a}\in\ker{\pi}\subseteq \oplus_{i=0}^rR[-u^i]$ be a
homogeneous element of degree $u\in\latv$. Then there are $b_i\in \field$,
$i=1,\ldots, r$ such that
\[
\underline{a}=\sum_{i=1}^r b_i\cdot(x^{u-u^i}e^{i}-x^{u-u^{i-1}}e^{i-1}),
\]
with $u-u^i\in\sigv$ whenever $b_i\not= 0$. In particular, $\underline{a}\in\img{\iota}$.
\end{lemma}
\begin{proof}
For the cases $r=1$ there is nothing to prove. We prove the lemma
explicitly for the case $r=2$. The methods used in that case are the same as
for the induction step in the general case, since we just show how to split off
the last summand.

Because $\underline{a}$ is homogeneous, every entry of it is a multiple of a
monomial. Hence, take $\underline{a}$ to be
\[
\underline{a}\ =\ (a_0x^{u-u^0},\ a_1x^{u-u^1},\ a_2x^{u-u^2}).
\]
Furthermore, assume $a_0\not=0\not=a_2$, otherwise we are done. This implies
$u-u^0\in\sigv$ and $u-u^2\in\sigv$. Together with the inequalities
of \autoref{gens:ineq} we get
\begin{gather*}
0\le\scalp{u-u^0,\rho^0}<\scalp{u-u^1,\rho^0}<\scalp{u-u^2,\rho^0}\\
0\le\scalp{u-u^2,\rho^1}<\scalp{u-u^1,\rho^1}<\scalp{u-u^0,\rho^1}.
\end{gather*}
Therefore, $u-u^1\in\sigv$ and we split $\underline{a}$ into the following
\begin{align*}
\underline{a}\ &=\ (a_0x^{u-u^0},\ a_1x^{u-u^1},\ a_2x^{u-u^2})\\
&=\ (a_0x^{u-u^0},\ (a_1-a_2) x^{u-u^1},0)\ +\ (0,\ a_2x^{u-u^1},\ -a_2x^{u-u^2}).
\end{align*}
The second summand is clearly an element of both $\ker{\pi}$ and $\img{\iota}$.
We conclude $a_1-a_2=-a_0$ and hence, we are done.
\end{proof}

\begin{proof}[Proof of \autoref{ses}]
It is already clear that $\pi$ is surjective.

We already know that $\iota$ is injective on the direct summands. Hence, an
homogeneous element of degree $u$ of the kernel of $\iota$ gives rise to an
equation
\[
\sum_{i=1}^r a_i(x^{u-u^{i-1}}e^{i-1}-x^{u-u^i}e^i),
\]
in $R^{\#\mingens{D}}$, with $a_i\in\field$ and $a_i=0$ for $x^u\notin
((x^{u^{i-1}})\cap (x^{u^i}))$. Now we factor out
\[
I\ =\ (x^u-1\ |\ u\in\sigv\cap\latv),
\]
and consider the equation over $R/I\cong \field$. But the vectors
$e^{i-1}-e^i\in\field^{\#\mingens{D}}$ are linearly independent and hence, $a_i=0$
for $i=1,\ldots,r$. Factoring out $I$ did not change the $a_i$, thus, $\iota$
must be injective.

All modules are $\latv$-graded and the maps are homogeneous of degree $0$.
Therefore $\ker{\pi}$ is a $\latv$-graded submodule of $\oplus_{i=0}^rR[-u^i]$. In
particular, $\ker{\pi}$ can be generated by homogeneous elements and thus,
\autoref{homogeneous:onto} implies $\img{\iota}=\ker{\pi}$.

The final claim is that $((x^{u^{i-1}})\cap (x^{u^i}))$ is
divisorial. Take the polyhedron
\[
(u^{i-1}+\sigv)\cap (u^i+\sigv).
\]
Its lattice points correspond exactly to the monomials in $((x^{u^{i-1}})\cap
(x^{u^i}))$. Furthermore it is the polyhedron of the divisor
$-\scalp{u^{i-1},\rho^0}[\rho^0]-\scalp{u^{i},\rho^1}[\rho^1]$.
\end{proof}

\begin{remark}
An alternative proof of \autoref{ses} can be given by using a 
modification of the criterion for exactness of \cite[Lemma 2.2]{bps:monres}.
One just needs to replace the expression of the least common multiple of two
monomials $x^u$ and $x^v$ by the intersection of the corresponding principal
ideals $(x^u)\cap(x^v)$. Using this approach, the above sequence is a
subcomplex of the Taylor resolution for the divisorial ideal $\module{\divi}$.
\end{remark}

Applying \autoref{ses} recursively, we can construct free resolutions of
$\divi$ up to any desired length. Taking into account the finiteness of the
class group of $\cqs$, we can encode the information of \autoref{ses} for
different divisors in a quiver.

For each $E^i$, \autoref{ses} gives a sequence
\[
0\to\bigoplus_{j=1}^{r(E^i)}E^{k_j}[-v^j]\inj\bigoplus_{j=0}^{r(E^i)}R[-u^j]\surj E^i\to 0,
\]
where $\{u^0,\ldots,u^{r(E^i)}\}=\mingens{E^i}$ and $v^j\in\latv$.

Let us introduce the labelled quiver $\resg$: It consists of an ordinary
quiver, i.e. a set of vertices $\resgvert$, a set of arrows $\resgar$ and two
functions $s,t:\resgar\to\resgvert$ returning the source and the target of an
arrow. Additionally we have a function $l:\resgar\to\latv$ equipping each arrow
with a label.

\begin{definition}
Take $\resg$ to be the quiver with vertices $\resgvert:=\{E^0,\ldots,
E^{n-1}\}$. For every direct summand $E^{k_j}[-v^j]$ in the above sequence we
add an arrow $a$ to $\resgar$ such that
\[
s(a) :=E^{k_j},\ t(a):= E^i,\ l(a):= v^j. 
\]
We do this for all such exact sequences for $i=0,\ldots,n-1$.
\end{definition}

Furthermore, for an arbitrary Weil divisor $\divi$ we define the sources of the
incoming arrows in $\resg$. This definition takes care of shifting all divisors
in the right way.

\begin{definition}
Let $E^i$ be the divisor linearly equivalent to $\divi$, i.e. $\divi=E^i[-u]$ for $u\in\latv$. Then we define
\[
\incoming{\divi}\ :=\ \{E^j[-v^j-u]\ |\ \exists\ a\in\resgar \mbox{ with }
s(a)=E^j,\ t(a)=E^i,\ l(a)=v^j\}.
\]
\end{definition}
The sequence of \autoref{ses} now becomes
\[
0\to \bigoplus_{G\in\incoming{\divi}}\module{G}\inj R^{\#\mingens{\divi}}\surj \divi\to 0,
\]
with the grading imposed by the generators of $\divi$ on the middle module.

\begin{example}\label{n7q3:ext_graph}
In the running example with $n=7$ and $q=3$ the quiver $\resg$ looks as follows:
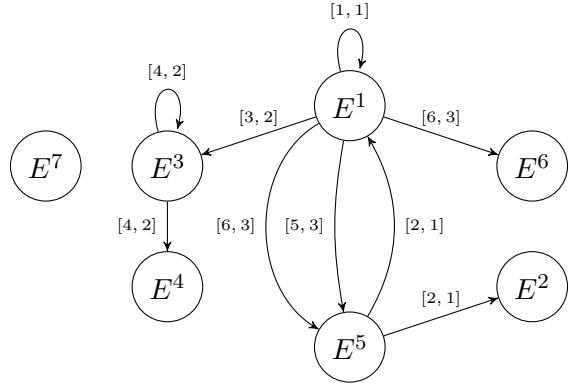
\[
\begin{tikzpicture}[scale=.8]
\def \radius {10pt};
\node[draw, circle] (E7) at (0,0) {$E^7$};
\node[draw, circle] (E1) at (5,1) {$E^1$};
\node[draw, circle] (E2) at (8,-2) {$E^2$};
\node[draw, circle] (E3) at (2,0) {$E^3$};
\node[draw, circle] (E4) at (2,-2) {$E^4$};
\node[draw, circle] (E5) at (5,-3) {$E^5$};
\node[draw, circle] (E6) at (8,0) {$E^6$};
\path[->] (E1) edge node[anchor=south, font=\tiny] {$[6,3]$} (E6);
\path[->] (E1) edge [bend right = 10] node[anchor=east, font=\tiny] {$[5,3]$} (E5);
\path[->] (E1) edge [bend right = 60] node[anchor=east, font=\tiny] {$[6,3]$} (E5);
\path[->] (E3) edge node[anchor=east, font=\tiny] {$[4,2]$} (E4);
\path[->] (E1) edge node[anchor=south, font=\tiny] {$[3,2]$} (E3);
\path[->] (E3) edge [loop above] node[anchor=south, font=\tiny] {$[4,2]$} (E3);
\path[->] (E5) edge node[anchor=south, font=\tiny] {$[2,1]$} (E2);
\path[->] (E1) edge [loop above] node[anchor=south, font=\tiny] {$[1,1]$} (E1);
\path[->] (E5) edge [bend right] node[anchor=west, font=\tiny] {$[2,1]$} (E1);
\end{tikzpicture}
\]
One immediately recognizes the first term in the sequence for $E^3$ given in \autoref{cqs:free_res:ex:res} via
\[
\incoming(E^3)\ =\ \{E^3[-[4,2]],\ E^1[-[3,2]]\}.
\]
One can use this construction to define a free resolution of $E^3$ recursively.
This resolution will even be minimal. However, the quiver $\resg$ is the more
elegant way of dealing with the infiniteness of these resolutions.
\end{example}

For the case of $\cqs$ being Gorenstein we recover a well-known result by
Eisenbud:

\begin{remark}
Building the quiver $\resg$ for the case $q=n-1$ one notices that it consists
of $\lfloor\frac{n-1}{2}\rfloor$ disjoint cycles of length $2$. For $n$ being
even, one of the cycles will have length one, i.e. it is a loop. In particular,
for every non-trivial divisor $\divi$ the direct sum at the beginning of the
exact sequence of \autoref{ses} has exactly one summand. This means that the
resolution of the divisorial ideal $\module{\divi}$ is $2$-periodic, which was
already observed by Eisenbud in \cite{eisenbud1980homological} and is due to
$\cqs$ being Gorenstein for this choice of $n$ and $q$ and $\module{\divi}$ being MCM.
\end{remark}

As a final remark, we can use the quiver $\resg$ to construct higher $\Ext^i$
and $\Tor_i$ recursively.

\begin{remark}\label{recursion}
For $\Ext$ and $\Tor$ we obtain the following formulas using \autoref{ses},
applying $\Hom(\bullet,\divz)$ and taking the long exact sequence of
cohomology:
\[
\Ext^{n+1}(\divi,\divz)\ =\ \bigoplus_{G\in\incoming{\divi}}\Ext^n(G,\divz)
\]
and
\[
\Tor_{n+1}(\divi,\divz)\ =\ \bigoplus_{G\in\incoming{\divi}}\Tor_n(G,\divz)
\]
for $n>0$.
\end{remark}
Thus, we only need to know how to compute $\Ext^1$ and $\Tor_1$.

\section{\texorpdfstring{$\Ext^1$}{Ext1}}\label{ext1:section}
Applying $\Hom(\bullet,\divz)$ to the short exact sequence of \autoref{ses} for
$\divi$, we may consider the long exact sequence of cohomology. The first part
gives a formula for $\Ext^1(\divi,\divz)$. In this section we will rephrase this
formula in combinatorial terms.

To a Weil divisor $\divi$ define the following set
\begin{definition}
\[
\below(\divi)\ :=\ \interior(\polyhedron{\divi})\backslash\bigcup_{u\in\mingens{\divi}} u+\interior(\sigv).
\]
\end{definition}

\begin{example}
We draw the sets $\below(E^i)$ for $\textcolor{red}{E^3}$ and $\textcolor{green}{E^2}$ in the running example with $n=7$ and $q=3$:
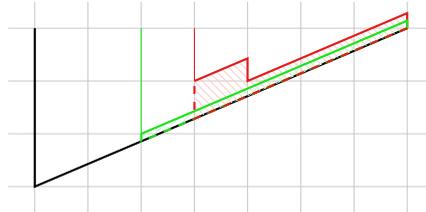
\[
\begin{tikzpicture}[scale=.7]
\draw[step=1.0,black!20,thin] (-0.5,-0.5) grid (7.5,3.5);
\draw[thick] (0,3) -- (0,0) -- (7,3);
\draw[color=green] (2,3) -- (2, 1);
\draw[color=red] (3,3) -- (3, 2);
\fill[pattern color=red!20, pattern=north west lines] (3,9/7) -- (3,2) -- (4,2+3/7) -- (4,2) -- (7,2+9/7) -- (7,3) -- cycle;
\fill[pattern color=green!20, pattern=north east lines] (2,6/7) -- (7,3) -- (7,3+1/7) -- (2,1) -- cycle;
\draw[color=red, thick] (3,2) -- (4,2+3/7) -- (4,2) -- (7,2+9/7) -- (7,3);
\draw[color=green, thick] (7,3) -- (7,3+1/7) -- (2,1);
\draw[color=green, thick, dashed] (2,1) -- (2,6/7) -- (7,3);
\draw[color=red, thick, dashed] (7,3) -- (3,9/7) -- (3,2);
\end{tikzpicture}
\]
Remember that the leftmost and the bottom boundary do not belong to the
$\below$ sets, this is indicated by the dashed lines. Since
$\textcolor{green}{E^2}$ only has two generators, its $\below$ set looks like a
parallelepiped. On the other hand, $\textcolor{red}{E^3}$ is generated by three
elements and hence, its $\below$ set has one 'dent'.
\end{example}

Let us now establish the link of $\below(\divi)$ with $\Ext^1(\divi,\divz)$.
\begin{prop}\label{ext1}
For two Weil divisors $\divi$ and $\divz$ define
\[
\ext(\divi,\divz)\ :=\ -(\below(\divi)-\vertex(\divz)).
\]
Then
\[
\Ext^1(\divi,\divz)\ =\ \short{\ext(\divi,\divz)}.
\]
\end{prop}

We start by proving a lemma about the support of the $\latv$-graded module
$\Hom(\divi, \divz)$.
\begin{lemma}\label{hom}
Denote by
\[
\hom(\divi,\divz)\ :=\ -\vertex(\divz)+\polyhedron{\divz}.
\]
Then
\[
\Hom(\divi,\divz)\ =\ \short{\hom(\divi,\divz)}.
\]
In particular, the module $\Hom(\divi,\divz)$ is a divisorial ideal.
\end{lemma}
\begin{proof}
We know that $\Hom(\divi,\divz)$ is $\latv$-graded and thus, generated by
homogeneous elements. Furthermore, we note that any non-zero map $\divi\to\divz$
must be injective, implying that the image of any non-zero element of $\divi$
completely determines the map. Hence, non-zero homogeneous maps $\divi\to\divz$
are given as $\field$-multiples of multiplication of $\divi$ by monomials
$x^u\in\field[\latv]$ such that $x^u\cdot\divi\subseteq\divz$. This leaves us with
determining all the lattice points
\[
\{u\in\latv\ |\ \polyhedron{\divi}+u\subseteq\polyhedron{\divz}\}.
\]
These are exactly the lattice points of $\hom(\divi,\divz)$.
\end{proof}

\begin{lemma}
Denote by $\{u^0,\ldots, u^r\}$ the generators $\mingens{\divi}$ of $\divi$. Furthermore write
\[
(u^{i-1}+\sigv)\cap(u^i+\sigv)=v^i+\sigv,\ v^i\in\latv_{\QQ},\ i=1,\ldots,r
\]
for the elements of $\incoming{\divi}$. Then
\[
\ext(\divi,\divz)\ =\ \bigcup_{i=1}^r\left(-v^i+\vertex(\divz)+\sigv\right)\backslash\left[(-u^0+\vertex(\divz)+\sigv)\cup(-u^r+\vertex(\divz)+\sigv)\right].
\]
\end{lemma}
\begin{proof}
Using the assumption we can write $\below(\divi)$ as
\[
\below(\divi)\ =\ \interior{\polyhedron{\divi}}\backslash\bigcup_{i=0}^r\left(u^i+\interior{\sigv}\right).
\]
Furthermore we note that
\[
(u^0-\sigv)\cap (u^r-\sigv)\ =\ \vertex(\divi)-\sigv.
\]
Now use the equations giving the $v^i$ to obtain
\[
\below(\divi)\ =\ \bigcup_{i=1}^r\left(v^i-\sigv\right)\backslash \left[(u^0-\sigv)\cup (u^r-\sigv) \right].
\]
Multiplying this with $(-1)$ and adding $\vertex(\divz)$ on both sides yields the desired formula.
\end{proof}

\begin{proof}[Proof of \autoref{ext1}]
Take the sequence of \autoref{ses} and apply $\Hom(\bullet, \divz)$ to it.
Considering the first part of the long exact sequence of cohomology we have
\begin{align*}
0 & \to\Hom(\divi,\divz)\to\Hom(R^{\#\mingens{\divi}},\divz)
\\
&\to\Hom(\bigoplus_{G\in\incoming{\divi}}G,\divz)\to\Ext^1(\divi,\divz)\to 0.
\end{align*}
Hence, we need to understand the quotient of
$\Hom(\oplus_{G\in\incoming{\divi}}G,\divz)$ by the image of
$\Hom(R^{\#\mingens{\divi}},\divz)$. Now we use \autoref{hom}:
\[
\Hom\left(\bigoplus_{G\in\incoming{\divi}}G,\divz\right)=\bigoplus_{G\in\incoming{\divi}}\Hom(G,\divz)
=\bigoplus_{G\in\incoming{\divi}}\short{\hom(G,\divz)}.
\]
Next we determine the image of $\Hom(R^{\#\mingens{\divi}},\divz)$ in this
direct sum. Take $u^i$ to be a generator of $\divi$, not at the boundary, i.e.
$0\not=i\not=r$. Then $D^i$ and $D^{i+1}$ are the only summands of
$\oplus_{G\in\incoming{\divi}}G$ mapping non-trivially to $R[-u^i]$ via the
following map
\[
\begin{array}{ccc}
D^i\oplus D^{i+1} & \to & R[-u^{i-1}]\oplus R[-u^i]\oplus R[-u^{i+1}],\\
(x^u,x^w) & \mapsto & (-x^{u-u^{i-1}},x^{u-u^{i}}-x^{w-u^{i}},x^{w-u^{i+1}})
\end{array}.
\]
Applying $\Hom(\bullet,\divz)$ to this map gives
\[
\begin{array}{ccc}
\Hom(R[-u^i],\divz) & \to & \Hom(D^i,\divz)\oplus\Hom(D^{i+1},\divz)\\
x^{u-u^i} & \mapsto & (x^u, -x^u)
\end{array}.
\]
Hence, dividing by the image of $\Hom(R[-u^i],\divz)$ means that the elements
$x^u\in\Hom(D^i,\divz)$ and $x^u\in\Hom(D^{i+1},\divz)$ get identified. This
explains why we take the union of the $-v^i+\sigma^\vee$.

Now consider $u^0$. Then the map $D^1\to R[-u^0]\oplus R[-u^1]$ yields
\[
\begin{array}{ccc}
\Hom(R[-u^0],\divz) & \to & \Hom(D^1,\divz)\\
x^{u-u^0} & \mapsto & x^u
\end{array}.
\]
Thus, all $x^u\in\Hom(D^1,\divz)$ such that $u + u^0 \in \polyhedron{\divz}$
are set to zero. One proceeds analogously for $u^r$. This explains the part
which is cut off and we are done.
\end{proof}

Let us briefly remark on the MCM'ness and sMCM'ness as mentioned in the introduction.

\begin{remark}\label{mcm}
We claim that $\Ext^i(\divi,\canonical)=0$ for $i>0$ and any $\divi$. By the
recursion formula it is enough to show that all $\Ext^1(\divi,\canonical)=0$.
Recall that the vertex of $\canonical$ is
$\vertex(\canonical)=[1,\frac{q+1}{n}]$.  Now proceed to compute
$\ext(\divi,\canonical) = \below(\divi)-\vertex(\canonical)$ in two steps:
First subtract $[1,\frac{q}{n}]$, second subtract $[0,\frac{1}{n}]$. One can then show
\[
(\below(\divi)-[1,\frac{q}{n}])\cap\latv\subseteq \below(\divi)\cap\latv.
\]
Now we argue that $\mingens{\divi}=\overline{\below(\divi)}\cap\latv$,
otherwise $\mingens{\divi}$ would not generate $\module{\divi}$.  Hence, the
closure of $\below(\divi)-[1,\frac{q}{n}]$ can only have lattice points on its
top and rightmost edges.  In particular, the lattice points of
$\below(\divi)-[1,\frac{q}{n}]$ do not lie in the `valleys' of
$\below(\divi)-[1,\frac{q}{n}]$, since they were at the `valleys' of
$\below(\divi)$.  Subtracting $[0,\frac{1}{n}]$ from
$\below(\divi)-[1,\frac{q}{n}]$ yields the desired result.
\end{remark}

\begin{remark}\label{ext1:gens}
Assume $\divi$ to be non-trivial. For the simplest case of $\module{0}=R$ in the second
argument we see
\[
\dim_{\field}\Ext^1_R(\divi,R)\ =\ \#\mingens{\divi} - 2.
\]
Thus, $\Ext^1_R(\divi,R)$ vanishes whenever $\module{\divi}$ is generated by
exactly two elements, thereby relating our construction to the result of Wunram
(\cite{wunram:cyclic}).
\end{remark}

\section{\texorpdfstring{$\Tor_1$}{Tor1}}\label{tor1:section}
Using the same strategy as for $\Ext$, we can derive a combinatorial
description of $\Tor_1$ as well. While $\Ext^1$ depended on the polyhedra of
global sections of the involved divisors exclusively, $\Tor_1$ needs the quiver
$\resg$ as an additional datum.

\begin{definition}
\[
\abelow(\divi)\ :=\ \polyhedron{\divi}\backslash\bigcup_{u\in\mingens{\divi}} u+\sigv.
\]
\end{definition}

\begin{prop}\label{tor1}
For two Weil divisors $\divi$ and $\divz$ let
\[
\tor(\divi,\divz)\ :=\ \abelow(\divi)+\vertex(\divz).
\]
Then
\[
\Tor_1(\divi,\divz)\ =\ \bigoplus_{\rundiv^i\in\incoming{\divi}}
\left(\bigoplus_{\rundiv^{ij}\in\incoming{\rundiv^i}}
\short{\tor(\rundiv^{ij},\divz)}
\right).
\]
\end{prop}
\begin{proof}
For a Weil divisor $\rundiv$, denote by $F_0(\rundiv)$ the $\latv$-graded module
$R^{\mingens{\rundiv}}$ in the middle of the sequence of \autoref{ses}. Then we can
build a free resolution of $\divi$ as follows:
\[
\bigoplus_{\rundiv^i\in\incoming{\divi}}\left(\bigoplus_{\rundiv^{ij}\in\incoming{\rundiv^i}}F_0(\rundiv^{ij})\right)\stackrel{d_2}{\to}
\bigoplus_{\rundiv^i\in\incoming{\divi}}F_0(\rundiv^i)
\stackrel{d_1}{\to} F_0(\divi)\surj \divi\to 0,
\]
resulting from repeatedly applying \autoref{ses}. Tensorizing this sequence with $\divz$ and then taking cohomology yields 
\[
\Tor_1(\divi,\divz)\ =\ \frac{\ker(d_1\otimes\id_{\divz})}{\img(d_2\otimes\id_{\divz})}.
\]
Tensorizing $R[-u]$ with $\divz$ yields $\divz[-u]$. Inserting this, we see that the kernel of $d_1\otimes\id_{\divz}$ is exactly
\[
\ker(d_1\otimes\id_{\divz})\ =\ 
\bigoplus_{\rundiv^i\in\incoming{\divi}}\left(\bigoplus_{\rundiv^{ij}\in\incoming{\rundiv^i}}
\left(
\short{\polyhedron{\rundiv^{ij}}+\vertex(\divz)}
\right)
\right).
\]
By construction it is enough to consider the single summands. In each of these we have to remove the image of $F_0(\rundiv^{ij})\otimes\divz$. Assume $\rundiv^{ij}$ to be generated by $\mingens{\rundiv^{ij}}=\{u^0,\ldots,u^r\}$. Then the image of $F_0(\rundiv^{ij})\otimes\divz$ is exactly
\[
\sum_{k=0}^r x^{u^k}\divz.
\]
Thus for the support we have
\[
(\polyhedron{\rundiv^{ij}}+\vertex(\divz))\backslash\bigcup_{k=0}^r(u^k+\divz)
\ =\
\vertex(\divz)+\left(
\polyhedron{\rundiv^{ij}}\backslash\bigcup_{k=0}^r(u^k+\sigv)
\right)
\]
and this equals $\tor(\rundiv^{ij},\divz)$, hence, finishing the proof.
\end{proof}

\section{The Matlis dual}\label{matlis:section}

\begin{definition}
We define the Matlis dual $\matlis{H}$ of an $R$-module $H$ to be
\[
\matlis{H}\ :=\ \Hom_R(H,\field[-\sigv\cap\latv]).
\]
One can check that in our setting, $\field[-\sigv\cap\latv]$ is exactly the
injective hull of $\field$ as an $R$-module.
\end{definition}
\begin{remark}\label{matlis:short}
As stated in \cite{sturmfelsmiller}, this means that for a $\latv$-graded
module $H$,
\[
(\matlis{H})_{-u}\ =\ \Hom_{\field}(H_u,\field).
\]
The $R$-multiplication on $\matlis{H}$ is then the transpose of the
$R$-multiplication on $H$.  In particular, this means that
$\matlis{\short{P}}=\short{-P}$.
\end{remark}

\begin{prop}\label{matlis:ext1}
The Matlis dual of $\Ext^1(\divi,\canonical-\divz)$ is
\[
\matlis{\Ext^1(\divi,\canonical-\divz)}\ =\ \short{\tor(\divi,\divz)}.
\]
\end{prop}

In order to prove this proposition, we will first have a closer look at the
sets $\below(\divi)$ and $\abelow(\divi)$. By construction they only differ at
the boundaries.  We will introduce a set $\universal(\divi)$, which is the
common core of $\below(\divi)$ and $\abelow(\divi)$, meaning that it links
these sets when intersecting with the lattice $\latv$. This set and its
behaviour under shifts by vertices $\vertex(\divz)$ is the key to understanding
\autoref{main:exttor}, the Matlis duality of $\Ext$ and $\Tor$.

\begin{definition}
For a Weil divisor $\divi$, define the \define{link}
of $\divi$:
\[
\universal(\divi)\ :=\ \below(\divi)\cap\left[\overline{\below(\divi)}+\vertex(\canonical)]\right].
\]
\end{definition}

\begin{example}
Let us construct $\universal(E^3)$. Again we have $\below(E^3)$ with the dashed lower and left edge. Now we shift it by 
\[
\vertex(\canonical)\ =\ [1,\frac{4}{7}]
\]
and take the closure. The intersection $\universal(E^3)$ is indicated in \textcolor{green}{green}.
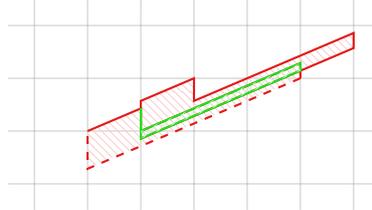
\[
\begin{tikzpicture}[scale=.7]
\draw[step=1.0,black!20,thin] (1.5,0.5) grid (8.5,4.5);
\fill[pattern color=red!20, pattern=north west lines] (3,9/7) -- (3,2) -- (4,2+3/7) -- (4,2) -- (7,2+9/7) -- (7,3) -- cycle;
\fill[pattern color=red!20, pattern=north west lines] (3+1,9/7+4/7) -- (3+1,2+4/7) -- (4+1,2+3/7+4/7) -- (4+1,2+4/7) -- (7+1,2+9/7+4/7) -- (7+1,3+4/7) -- cycle;
\draw[color=red, thick] (3,2) -- (4,2+3/7) -- (4,2) -- (7,2+9/7) -- (7,3);
\draw[color=red, thick, dashed] (7,3) -- (3,9/7) -- (3,2);
\draw[color=red, thick] (3+1,9/7+4/7) -- (3+1,2+4/7) -- (4+1,2+3/7+4/7) -- (4+1,2+4/7) -- (7+1,2+9/7+4/7) -- (7+1,3+4/7) -- cycle;
\draw[color=green, thick] (3+1,9/7+4/7) -- (3+1+3,9/7+4/7+9/7) -- (7,2+9/7) -- (4,2) -- (4,2+3/7) -- cycle;
\fill[pattern color=green!20, pattern=north east lines] (3+1,9/7+4/7) -- (3+1+3,9/7+4/7+9/7) -- (7,2+9/7) -- (4,2) -- (4,2+3/7) -- cycle;
\end{tikzpicture}
\]
\end{example}
Alternatively one can define $\universal(\divi)$ using the $\abelow$ set. One just has to close the first set in the intersection. This is the first statement in the following proposition.

\begin{prop}\label{torext}
Given $\divi$ and $\divz$ two Weil divisors on $\cqs$, we have
\begin{enumerate}
\item
$\universal(\divi)\ =\ \overline{\abelow(\divi)}\cap[\abelow(\divi)+\vertex(\canonical)]$;
\item
$[\below(\divi)+\vertex(\divz)]\cap\latv\ =\ [\universal(\divi)+\vertex(\divz)]\cap\latv$; and
\item
$[\abelow(\divi)+\vertex(\divz)]\cap\latv\ =\ [\universal(\divi)+\vertex(\divz)-\vertex(\canonical)]\cap\latv$.
\end{enumerate}
\end{prop}
\begin{proof}
It is already clear that the closures of $\below$ and $\abelow$ are equal. The
complementary boundary containments then yield the first formula.

The proofs of the second and third claim are very similar, hence, we will only
prove the second claim.

The containment of $\universal(\divi)+\vertex(\divz)$ in $\below(\divi)+\vertex(\divz)$ is trivial.

By construction, $\below(\divi)$ is contained in $\polyhedron{\divi}$, even in the interior. Let
\[
\polyhedron{\divi}\ =\ \{u\in\latv_{\QQ}\ |\ \scalp{u,\rho^0}\ge a_{\divi},\ \scalp{u,\rho^0}\ge b_{\divi}\},
\]
with $a_{\divi},b_{\divi}\in\ZZ$, i.e. $\divi = -a_{\divi}\cdot[\rho^0]-b_{\divi}\cdot[\rho^1]$. Similarly, let
$\divz = -a_{\divz}\cdot[\rho^0]-b_{\divz}\cdot[\rho^1]$ for some $a_{\divz},b_{\divz}\in\ZZ$. Hence
\[
\below(\divi)+\vertex(\divz)\subseteq \polyhedron{\divi}+\vertex(\divz)=
\left\{u\in\latv_{\QQ}\ |\ 
\begin{array}{ccc}\scalp{u,\rho^0} &\ge& a_{\divi}+a_{\divz},\\ \scalp{u,\rho^1}&\ge& b_{\divi}+b_{\divz}
\end{array}
\right\}.
\]
Now we want lattice points in the interior of the right hand side. Taking
$u\in\interior(\polyhedron{\divi}+\vertex(\divz))$, $u$ evaluates to an
integer with both $\rho^0$ and $\rho^1$. Since it is not on the boundary of
$\polyhedron{\divi}+\vertex(\divz)$, we have
\[
\scalp{u,\rho^0}\ge a_{\divi}+a_{\divz}+1,\mbox{ and }\scalp{u,\rho^1}\ge b_{\divi}+b_{\divz}+1.
\]
This corresponds exactly to adding the vertex of the canonical divisor $\canonical=-[\rho^0]-[\rho^1]$. We obtain
\[
[\below(\divi)+\vertex(\divz)]\cap\latv\subseteq \polyhedron{\divi}+\vertex(\divz)+\vertex(\canonical).
\]
Intersecting the right hand side with $\below(\divi)+\vertex(\divz)$ preserves
this relation. As a final step, we note that we can replace
$\polyhedron{\divi}$ by $\overline{\below(\divi)}$.  One then recognizes this
intersection as $\universal(\divi)+\vertex(\divz)$.
\end{proof}

\begin{proof}[Proof of \autoref{matlis:ext1}]
Inserting the formula for $\tor$ and $\ext$ in terms of $\below(\divi)$ and
$\abelow(\divi)$, we use the formula of \autoref{torext} to obtain:
\begin{align*}
\tor(\divi,\divz)\cap\latv &= [\abelow(\divi)+\vertex(\divz)]\cap\latv\\
&= [\universal(\divi)+\vertex(\divz)-\vertex(\canonical)]\cap\latv\\
&= [\below(\divi)-\vertex(\canonical-\divz)]\cap\latv\\
&= -\ext(\divi, \canonical-\divz)\cap\latv.
\end{align*}
Inserting \autoref{matlis:short} finishes the proof.
\end{proof}

\section{Main theorems}
\begin{theorem}\label{ext1:symmetric}\label{symmetry}
Let $\divi$ and $\divz$ be two Weil divisors on a cyclic quotient singularity. Then
\[
\Ext^1(\divi,\canonical-\divz)\ =\ \Ext^1(\divz,\canonical-\divi).
\]
\end{theorem}
\begin{proof}
We want to prove the equality of the lattice points of the respective supports, i.e.
\[
\ext(\divi,\canonical-\divz)\cap\latv\ =\ \ext(\divz,\canonical-\divi)\cap\latv.
\]
The trick is to show that there are no lattice points in the respective
complements. Since this is symmetric we will just show
\[
\left[\ext(\divz,\canonical-\divi)\backslash \ext(\divi,\canonical-\divz)\right]\cap\latv = \emptyset.
\]
We reverse the sign and insert the definition of $\below(\divi)$ to get
\begin{align*}
-\ext(\divi,\canonical-\divz) &= \vertex(-\canonical+\divz)+\below(\divi)\\
&=(\vertex(-\canonical+\divi+\divz)+\interior(\sigv))\backslash\bigcup_{i=0}^r(\vertex(\divz-\canonical)+u^i+\interior{\sigv}),
\end{align*}
where $\{u^0,\ldots,u^r\}=\mingens{\divi}$.
Both $-\ext(\divi,\canonical-\divz)$ and $-\ext(\divz,\canonical-\divi)$ result
from $\vertex(-\canonical+\divi+\divz)+\interior(\sigv)$ by cutting off certain
pieces at the top. Inserting the equation for $-\ext(\divi,\canonical-\divz)$, we can
rephrase this as an intersection, i.e.
\begin{align*}
-\ext(\divz,\canonical-\divi)&\backslash -\ext(\divi,\canonical-\divz)\\
&=\ -\ext(\divz,\canonical-\divi)\cap\bigcup_{i=0}^r(\vertex(\divz)-\vertex(\canonical)+u^i+\interior{\sigv}).
\end{align*}
The set $-\ext(\divz,\canonical-\divi)$ is a shift of $\below(\divz)$.  Thus,
we can replace $\vertex(\divz)+u^i+\interior{\sigv}$ by $u^i+\below(\divz)$ in
the above intersection. Hence, 
\begin{align*}
-\ext(\divz,\canonical-\divi)\backslash -\ext(\divi,\canonical-\divz)\ &\subseteq\ \bigcup_{i=0}^r(-\vertex(\canonical) + u^i+\below(\divz))\\
&=\ \bigcup_{i=0}^r-\ext(\divz[-u^i],\canonical).
\end{align*}
Since $\divz$ is $\QQ$-Cartier and hence, MCM, we know that
$\Ext(\divz,\canonical)=0$, which stays true under $\latv$-shifts of $\divz$.
Thus, $-(\ext(\divz[-u^i],\canonical)\cap\latv)=\emptyset$ for all
$i=0,\ldots,r$, and we are done.
\end{proof}

Finally, we show the duality of $\Ext$ and $\Tor$.

\begin{theorem}\label{main:exttor}\label{exttor}
Let $\divi$ and $\divz$ be two Weil divisors on a cyclic quotient singularity. Then
\[
\Ext^{i+2}(\divi,\canonical-\divz)\ =\ \matlis{\Tor_i(\divi,\divz)}
\]
for $i>0$.
\end{theorem}
\begin{proof}
Using the recursion of \autoref{recursion} on both sides this is a consequence of \autoref{ext3tor1} below.
\end{proof}

\begin{lemma}\label{ext3tor1}
Given two Weil divisors $\divi$ and $\divz$, we have the following equality of $R$-modules:
\[
\Ext^{3}(\divi,\canonical-\divz)\ =\ \matlis{\Tor_1(\divi,\divz)}.
\]
\end{lemma}

\begin{proof}
First, we use the description of $\Tor_1$ developed in \autoref{tor1}. Then we
insert the formula for the Matlis dual of $\Ext^1$ of \autoref{matlis:ext1}.
\begin{align*}
\Tor_1(\divi,\divz)\ &=
\ \bigoplus_{\rundiv^i\in\incoming{\divi}}
\left(\bigoplus_{\rundiv^{ij}\in\incoming{\rundiv^i}}
\short{\tor(\rundiv^{ij},\divz)}
\right)\\
&\cong\ \bigoplus_{\rundiv^i\in\incoming{\divi}}
\left(\bigoplus_{\rundiv^{ij}\in\incoming{\rundiv^i}}
\matlis{\Ext^1(\rundiv^{ij},\canonical-\divz)}
\right).
\end{align*}
Applying the recursion formula of \autoref{recursion} for $\Ext^3$, we obtain the desired result.
\end{proof}

Note that although \autoref{exttor} shows the symmetry of \autoref{symmetry}
for $i>2$, it does not generalize \autoref{symmetry}, as it does not imply the
$i=1$ case.  In particular, the case $i=2$ remains open.

\printbibliography

\end{document}